\documentclass[12pt]{amsart}
\usepackage{amsfonts}
\usepackage{enumerate}
\usepackage{helvet,courier}
\usepackage{mathptmx,amsmath}
\usepackage{type1cm}
\DeclareMathAlphabet{\mathcal}{OMS}{cmsy}{m}{n}
\DeclareMathAlphabet{\mathbbold}{U}{bbold}{m}{n}  
\usepackage{amsthm}
\usepackage{graphicx}       
\usepackage{multicol}
\usepackage{mathrsfs}
\usepackage{amssymb}
\usepackage{verbatim}
\usepackage{esint}

\usepackage[usenames,dvipsnames]{xcolor}

\theoremstyle{plain}
\newtheorem{thm}{Theorem}[section]
\newtheorem{lm}[thm]{Lemma}
\newtheorem{cor}[thm]{Corollary}
\newtheorem{prop}[thm]{Proposition}

\theoremstyle{remark}
\newtheorem{rmk}{Remark}

\theoremstyle{definition}

\newcommand{\bnu}{\begin{enumerate}}
\newcommand{\enu}{\end{enumerate}}
\newcommand{\bpf}{\begin{proof}}
\newcommand{\epf}{\end{proof}}

\newcommand{\qq}{\qquad}

\newcommand{\al}{\alpha}
\newcommand{\be}{\beta}
\newcommand{\ga}{\gamma}

\newcommand{\om}{\omega}

\newcommand{\la}{\lambda}

\newcommand{\ep}{\epsilon}

\newcommand{\si}{\sigma}
\newcommand{\tht}{\theta}

\newcommand{\vp}{\varphi}
\newcommand{\de}{\delta}

\newcommand{\bbr}{\mathbb{R}}

\newcommand{\rn}{\mathbb{R}^n}

\newcommand{\f}{\frac}

\newcommand{\p}{\partial}
\newcommand{\nf}{\infty}

\newcommand{\tf}{\tfrac}
\newcommand{\wh}{\widehat}

\newcommand{\wtd}{\widetilde}

\newcounter{question}
\newcommand{\qt}{%
        \stepcounter{question}%
        \thequestion}
\newcommand{\bq}{\fbox{Q\qt}\ }
\newcounter{project}

\newcommand{\mm}{\mathcal M}

\newcommand{\vu}{{\vec\mu}}

\allowdisplaybreaks

\makeindex         

\begin{document}

\author[L. Grafakos]{Loukas Grafakos}
\address{Department of Mathematics, University of Missouri, Columbia MO 65211, USA}
\email{grafakosl@missouri.edu}

\author[D. He]{Danqing He}
\address{Department of Mathematics, Sun Yat-sen  University, Guangzhou, 510275, P. R. China}
\email{hedanqing@mail.sysu.edu.cn}

\author[P. Honz\'ik]{Petr Honz\'ik}
\address{MFF UK,
Sokolovska 83,
Praha 7,
Czech Republic}
\email{honzik@gmail.com}

\thanks{The first author would like to acknowledge the support  of the Simons Foundation
and of the University of Missouri Research Board and Research Council. 
The second author was supported by NNSF of China (No. 11701583),  the Guangdong Natural Science Foundation
(No. 2017A030310054), and the
Fundamental Research Funds for the Central Universities (No. 17lgpy11).
The third author  
 was supported by GA\v CR P201/18-07996S}
\thanks{2010 Mathematics Classification Number 42B15, 42B20, 42B25}

\title[Bilinear maximal operators]{Maximal operators associated with bilinear multipliers of limited decay}
\date{}
\maketitle
\begin{abstract}
Results analogous to those 
proved by Rubio de Francia \cite{RubiodeFrancia1986} are obtained 
for  a class of   maximal functions  
formed by dilations of bilinear multiplier operators of limited decay. We focus our attention 
to $L^2\times L^2\to L^1$ estimates. We discuss two applications:  the 
 boundedness of the bilinear maximal Bochner-Riesz operator  
and   of the bilinear spherical maximal operator. For the latter we improve the known results in 
\cite{Barrionuevo2017} by reducing the dimension restriction from $n\ge 8$ to $n\ge 4$.
\end{abstract}

%\tableofcontents
%Goal 1. Write the whole proof dealing with only the bilinear spherical maximal operator.

\section{Introduction}

Coifman and Meyer \cite{Coifman1975, Coifman1978, Coifman1978b} initiated the study of bilinear 
singular integrals and set the cornerstone of a theory that has recently flourished   in view of the breakthrough results in \cite{Lacey1997, Lacey1999} and of the 
foundational work in \cite{Grafakos2002,  Kenig1999}.
The study of multipliers of limited decay in the bilinear setting, such as of 
Mihlin-H\"ormander type, was initiated in \cite{Tomita2010} and pursued further in  
  \cite{FT, GMT, GS, MT} and other works. Many of these results
have found weighted extensions in terms of the natural multilinear   weights introduced in \cite{Lerner2009}. 
Meanwhile,  the simple characterization of multipliers bounded on $L^2$ 
does  not have a bilinear analogue; see  \cite{Benyi2003} and \cite{Grafakos2018}.

In this work we investigate the $L^2\times L^2\to L^1$ 
boundedness of maximal operators related to bilinear multipliers with limited decay.  
%{\color{blue} 
This line of investigation was 
motivated   by the study of the %bilinear maximal operator of this type is the 
bilinear spherical maximal operator  introduced in   \cite{Geba2012} and  further studied in \cite{Barrionuevo2017}; 
%, and 
%the bilinear maximal Bochner-Riesz operator, the maximal function 
%related to bilinear Bochner-Riesz means studied by 
%\cite{Grafakos2006a, Bernicot2015a, Jeong2017}.
another bilinear version of the  spherical maximal operator is studied in \cite{GIKL}.

The spherical maximal operator was   shown to be $L^p$ bounded  by 
Stein \cite{Stein1976} in dimensions $n\ge 3$ (see also \cite[Chapter XI]{Stein1993}) 
but its planar version ($n=2$) 
 was completed by Bourgain \cite{Bourgain1986}.  
Rubio de Francia  \cite{RubiodeFrancia1986} introduced a
different approach to study this operator in dimensions $n\ge 3$ 
and proved the following theorem concerning 
  general maximal functions  that include 
the spherical maximal operator.
%}
 
\smallskip

\noindent{\bf Theorem (\cite[Theorem B]{RubiodeFrancia1986}).}
{\it 
Let $s$ be an integer with $s>n/2$, let   $a>1/2$, and suppose that
$m $  is a function of class $\mathcal C^{s+1}(\mathbb R^n)$ that satisfies 
$$
|D^\al m(\xi)|\le C|\xi|^{-a}\qq \text{for all $|\al|\le s+1$.}
$$
 Then, $T_m^*(f)=\sup_{t>0}|(m(t\cdot)\wh f)^\vee |$ is bounded on $L^p(\mathbb R^n)$ for
$$
q_a=\f{2n}{n+2a-1}<p<\f{2n-2}{n-2a}=r_a
$$
(with the understanding that $q_a=1$ if $a > (n + 1)/2$ and $r_a=\nf$ if $a \ge n/2$).

}
Here $\wh f $ is the Fourier transform of 
$f$ given by $\wh f(\xi) = \int_{\rn} f(x) e^{-2\pi i x\cdot \xi} dx.$

\medskip

In this paper we are concerned with maximal operators formed by 
dilations of  bilinear multiplier operators  of the form 
$$
T_m(f,g)(x)=\sup_{t>0}\Big|\int_{\rn}\int_{\rn}m(t\xi,t\eta)\wh f(\xi)\wh g(\eta)e^{2\pi ix\cdot (\xi+\eta)}d\xi d\eta\Big|
$$
for all Schwartz functions $f$ and $g$ on $\rn$.  Our main result is
the  following theorem, which presents a bilinear 
analogue of the aforementioned  result of Rubio de Francia. 

\begin{comment}
For a bilinear (or bi-sublinear) operator $T$ defined for pairs of functions in Lebesgue spaces, we 
define the norm
$$
\|T\|_{L^{p_1}\times L^{p_2}\to L^p}
=\sup_{\|f\|_{L^{p_1}}\le 1}\sup_{\|g\|_{L^{p_2}}\le 1}
\|T(f,g)\|_{L^p}.
$$
If this quantity is finite, then we say that $T$ is bounded from $L^{p_1}\times L^{p_2}$ to $ L^p$. 
\end{comment}

%\section{A bilinear version of Rubio de Francia theorem}

%Our first main result is a bilinear version of 
%,which can also be regarded as a  continuous analogue of the main result of \cite{Grafakos2016a}.
%This result has a direct application in the boundedness of bilinear spherical maximal operator.

\begin{thm}\label{04071}
Let  $a>\tf n2+1$.
Suppose that $m(\xi,\eta)\in C^{\nf}(\bbr^{2n})$ satisfies  
$$
|\p^\be m|\le C_\be |(\xi,\eta)|^{-a}
$$
 for all $|\be|\le  [\tf n2]+2$, where $[\tf n2]$ is the integer part of $\tf n2$.
Define
$$
S_t(f,g)(x)=\int_{\bbr^{2n}}m(t\xi,t\eta)\wh f(\xi)\wh g(\eta)e^{2\pi ix\cdot(\xi+\eta)}
d\xi d\eta.
$$
Then
the bilinear maximal operator defined by
\begin{equation}\label{calM}
\mathcal M(f,g)=\sup_{t>0}|S_t(f,g)|
\end{equation}
is bounded from $L^2(\bbr^n)\times L^2(\rn)$
to $L^1(\rn)$.

\end{thm}

%\bq * Where do we use the condition $|\be|\le n+2$?

%\bq How small can the number $\be$ be? $\tf n2+2$?

%\bq How small $a$ could be? How about $\tf {n+1}2$?

%An important motivation for studying bilinear maximal operator of this type is the 
%bilinear spherical maximal operator, in which case
% the multipler $m=\wh{d\si}$ satisfies $|m(\xi,\eta)|\le C|(\xi,\eta)|^{-(n-1/2)}$, where 
%$\si$ is surface measure on $\mathbb S^{2n-1}$.  
%This operator was introduced in   \cite{Geba2012} and  further studied in \cite{Barrionuevo2017}.

%\begin{rmk}
%It seems that $a>\tf{n+1}2$ is possible if we can
%verify by the bilinear square function theory (to be studied).

%Is this sharp in general?

%\end{rmk}

%\begin{rmk}
%The bilinear maximal operator $T^\#$ studied in [BURIANKOVA-Honzik]
%may be regarded as 
%$\sup_{k\in \mathbb Z} |S_{2^k}(f,g)|$ with
%$m(\xi,\eta)=\sum_{i\ge 0} \wh{K^0}(2^i(\xi,\eta))$.

%In this case we need only essentially $a>0$ because of the dilation invariant structure and that $\wh{K^0}\in L^2$ when $\Om\in L^2$.

%\end{rmk}

\medskip

In studying linear and bilinear spherical maximal operators, we often decompose the multiplier
 $m=\sum_{j=0}^\nf m_j$ with $m_j=m\psi_j$ for smooth  bumps $\psi_j$ supported in 
 annuli $|(\xi,\eta)|\approx 2^j$, $j\ge 1$ and $\psi_0$ supported in a neighborhood of the origin.
%This inspires us to prove the following general result by characterizing the multipliers
%in terms of H\"ormander type conditions.
We recall the  Sobolev space $L^r_s$ of all functions $g$ with $\|(I-\Delta)^{s/2}g\|_{L^r}<\nf$, 
where $\Delta$ is the usual Laplaciand and $s>0$. 

%\bq  * Do we have the case $r=4$?

Motivated by H\"ormander type conditions, we obtain 
Theorem~\ref{04071}  as a consequence of the following more general result,
which is the main contribution of this paper.

\begin{thm}\label{04141}
Let  $\la>1$, $1< r\le 4$, $s> \tf{2n}r+1, j\ge 1$.
Suppose that for each $j\in \mathbb N$, $M_j(\xi,\eta)$ is a multiplier supported in
$$
\{(\xi,\eta)\in\bbr^{2n}:\ 2^{j-1}\le |(\xi,\eta)|\le 2^{j+1}\},
$$
that satisfies
\begin{equation}\label{e04141}
\|M_j\|_{L^r_{s}(\bbr^{2n})}\le A 2^{-\la j}.
\end{equation}
% and $\wh\psi$ is a smooth function supportedin the unit annulus.
Let
$$
S_t(f,g)(x)=\int_{\bbr^{2n}}\sum_{j\ge0}M_j(t\xi,t\eta)\wh f(\xi)\wh g(\eta)e^{2\pi ix\cdot(\xi+\eta)}
d\xi d\eta .
$$
Then the   maximal operator  
$$
T(f,g)=\sup_{t>0}|S_t(f,g)|
$$
is bounded from $L^2(\bbr^n)\times L^2(\rn)$
to $L^1(\rn)$ with   bound a constant multiple of $ A$.
\end{thm}

%\bq * Check the case when $j$ is small.

%\bq  * The relation between $\|T\|$ and $A$ in \eqref{e04141}?

%\bq Is $s>\tf n2+1$ enough here?

%\bq  * Does the result hold for $1\le r\le 2$?

We prove Theorem~\ref{04141}  in Section~\ref{06091}. Below we   derive 
Theorem~\ref{04071}.

%We are now at a position to prove Theorem~\ref{04071}.

\bpf[Proof of Theorem~\ref{04071} assuming Theorem~\ref{04141}]

We fix a smooth function $\wh\vp$
  supported in $B(0,2)$ whose value is $1$ in the unit ball,
and define 
$$
\wh\psi(\cdot)=\wh\vp(2^{-1}\cdot)-\wh\vp(\cdot) 
$$
and
$$
m_j(\xi,\eta)=m(\xi,\eta)\wh\psi(2^{-j}(\xi,\eta))
$$
 for $j\ge 1$,
and $m_0=m-\sum_{j\ge 1}m_j$.

Then $m_0$ is a compactly supported smooth function, so the corresponding bilinear maximal operator
$$
T_{m_0}^*(f,g)(x)=\sup_{t>0}\Big|\int_{\bbr^{2n}}m_0(t\xi,t\eta)\wh f(\xi)\wh g(\eta)
e^{2\pi ix\cdot(\xi+\eta)}d\xi d\eta \Big|
$$
is bounded by $CM(f)M(g)$, where $M$ is the Hardy-Littlewood maximal function.
So $T_0$ is bounded from $L^{p_1}(\rn)\times L^{p_2}(\rn)$
to $L^p(\rn)$ for all $1<p_1,p_2<\nf$ and $\tf1p=\tf1{p_1}+\tf1{p_2}$.

Let $r=4$, then
$\|m_j\|_{L^4_s}\le C2^{-ja}2^{jn/2}$.
Hence $m_j$ 
satisfies conditions of Theorem~\ref{04141} with the decay $\la=a-\tf n2>1$. Theorem~\ref{04141}
then implies that the bilinear maximal operator \eqref{calM} is bounded from $L^2\times L^2$ to $L^1$.
\epf

%\bq What is the range of $p_1,p_2,p$ where we have the boundedness of $T$ in Theorem~\ref{04071} and Theorem~\ref{04141}?

%\bq More boundedness in higher dimensional cases?

As an application of Theorem~\ref{04071}, we improve the known results concerning the boundedness of the bilinear spherical maximal operator. It was shown in 
\cite{Barrionuevo2017} that this 
operator is bounded from $L^2(\rn) \times L^2(\rn)$
to $L^1(\rn)$ for $n\ge 8$. Here we reduce the dimension restriction to $n\ge 4$. 

\begin{thm}\label{12093}
Let $m_\al(\xi,\eta)=\tf{J_{n+\al-1}(2\pi|\xi,\eta|)}{|(\xi,\eta)|^{n+\al-1}}$ for $\al\in\mathbb R$,
then the bilinear maximal operator $M_\al$ defined by 
$$
M_\al(f,g)=\sup_{t>0}\bigg|\iint_{\bbr^{2n}}m_\al(t\xi,t\eta)\wh f(\xi)\wh g(\eta)e^{2\pi ix\cdot(\xi+\eta)}
\, d\xi d\eta\bigg|
$$ 
is bounded from $L^2\times L^2$
to $L^1$ when $n>3-2\al$.

In particular, for $\alpha =0$,   the
 bilinear spherical maximal operator 
\begin{equation}\label{BSMOOO}
\mathcal M_0(f,g)=\sup_{t>0}\bigg|\int_{\mathbb S^{2n-1}} f(x-t\theta) g(x-t\phi) 
d\sigma (\theta,\phi)\bigg|
\end{equation}
  is bounded from $L^2(\bbr)
\times L^2(\bbr)$ to $L^1(\bbr)$ when $n\ge 4$.
\end{thm}

\bpf
The function 
$m_\al$ satisfies the conditions of Theorem~\ref{04071}
with $a=n+\al-\tf12$. Hence when $n>3-2\al$, we  obtain the $L^2\times L^2 \to L^1$ boundedness of $M_\al$. Now recall that the    $(2n-1)$-dimensional surface measure $\sigma$  satisfies  
  $\wh{d\si}=m_0$, i.e.,     $m_\alpha$  with $\al=0$. 
Hence the bilinear spherical maximal operator \eqref{BSMOOO} is bounded from $L^2(\bbr^n)
\times L^2(\bbr^n)$ to $L^1(\bbr^n)$ when $n\ge 4$.
\epf

It was pointed out in \cite{Barrionuevo2017} that $L^2\times L^2 \to L^1$ boundedness 
fails in dimension  $n=1$. As of this writing,   we 
 are uncertain about the behavior of this operator in dimensions $n=2,3$.

%\bq Do we have more pints for bilinear spherical maximal?

%\bq Can we improve $\de_n$ to $n-1$ as we conjectured if we prove the bilinear square function theory.

%\bq *  A counterexample for $T$ in Theorem~\ref{04071}.

We discuss another application of Theorem~\ref{04141}  concerning the 
 bilinear maximal Bochner-Riesz means in Section 4.

%Our main theorem in this section is as follows.
\begin{comment}\label{06151}
When $\la>1$, $T_*^\la$
is bounded from $L^{p_1}\times L^{p_2}$
to $L^p$ when
$p>(2n-3)/(2n-5+2\la)$,
$p_1,\ p_2>(4n-6)/(2n-5+2\la)$
and $1/p=1/p_1+1/p_2$. In particular,
$T^*$ is bounded from $L^2(\rn)\times L^2(\rn)$ to $L^1(\rn)$
for all $\la>1$.

\end{comment}

%\newpage

\section{Wavelet decomposition}
 We   use  the wavelet decomposition of  multipliers as in \cite{Grafakos2015}. So we need 
 to introduce the tensor type wavelets due to \cite{Daubechies1988}, and the exact form we use here can be found
%We cite the following lemmas without proofs, which can be found 
in \cite{Triebel2006}.% and \cite{GraHeHon}.%Then we decompose $m_j$ so that $m_j=m_j^1+m_j^2+m^3_j$,where $m_j^1(\xi,\eta)=m_j(\xi,\eta)\rho(\tf1{Mj}(\log_2\tf{|\xi|}{|\eta|}))$.%, then we have a smooth   decomposition of $m_j$ with % into diagonal part $m_j^{1}$
%and off-diagonal part $m_j^2$  so that

% For $m_j^1$, $m_j^2$ and $m^3_j$, we use the wavelet decomposition  as in \cite{GraHeHon}.We cite the following lemmas without proofs, which can be found in \cite{Tr1} and \cite{GraHeHon}.

\begin{lm}[{\cite[Section 1.7.3]{Triebel2006}}]\label{wave}
For any fixed $k\in \mathbb N$ there exist real compactly supported functions $\psi_F,\psi_M\in \mathcal C^k(\mathbb R)$, 
%{\color{red}
which
satisfy that $\|\psi_F\|_{L^2(\mathbb R)}=\|\psi_M\|_{L^2(\mathbb R)}=1$
and $\int_{\mathbb R}x^{\al}\psi_M(x)dx=0$ for $0\le\al\le k$,   %}
such that,  if $\Psi^G$ is defined by 
$$
\Psi^{G}(\vec x\,)=\psi_{G_1}(x_1)\cdots \psi_{G_{2n}}(x_{2n}) 
$$
for   $G=(G_1,\dots, G_{2n})$ in the set     
$$
 \mathcal I :=\Big\{ (G_1,\dots, G_{2n}):\,\, G_i \in \{F,M\}\Big\}  \, , 
%\mathcal I^*= \mathcal I\setminus \{(F,\dots , F)\}
$$
then the  family of 
functions 
$$
\bigcup_{\vec \mu \in  \mathbb Z^{2n}}\bigg[  \Big\{   \Psi^{(F,\dots, F)} (\vec  x-\vec \mu  )  \Big\} \cup \bigcup_{\ga=0}^\nf
\Big\{  2^{\ga n}\Psi^{G} (2^{\ga}\vec x-\vec \mu):\,\, G\in \mathcal I\setminus \{(F,\dots , F)\}  \Big\}  
  \bigg]
$$
forms an orthonormal basis of $L^2(\mathbb R^{2n})$, where $\vec x= (x_1, \dots , x_{2n})$.  
\end{lm}

For simplicity, we use often below $\om(\xi,\eta)=\om_{k,l}(\xi,\eta)$ to denote the wavelet 
$2^{\ga n}\Psi^{G} (2^{\ga}(\xi,\eta)-(k,l))$ when the dilation factor $\ga$ is fixed.
Moreover we may write $\om_{k,l}(\xi,\eta)=\om_{1,k}(\xi)\om_{2,l}(\eta)$, where 
\begin{equation}\label{defomega1k}
\om_{1,k}(\xi)=2^{\ga n/2}\psi_{G_1}(2^\ga\xi_1-k_1)\cdots \psi_{G_{ n}}(2^\ga\xi_{n}-k_n)
\end{equation}
 and
$\om_{2,l}(\eta)$ is defined in an obvious similar way.
For a good function $m$, we denote by $a_{k,l}$ as the inner product $\langle m,\om_{k,l}\rangle$ of $m$ and $\om_{k,l}$.

Let $F^s_{r,q}(\mathbb R^{2n})$
and $f^s_{r,q} $ be the Triebel-Lizorkin spaces of functions and sequences, respectively; see \cite[Sections 2.2 and 2.3]{Grafakos2014a}. 
To characterize general function spaces, we need the following lemma.

\begin{lm}[{\cite[Theorem 1.64]{Triebel2006}}]\label{08311}
Let $0<r<\nf,\ 0<q\le\nf, \ s\in\mathbb R$, and
for $\ga\in \mathbb N$ and $\vec\mu\in\mathbb N^{2n}$ let $\chi_{\ga\vec\mu}$
be the characteristic function of the cube $Q_{\ga\vec\mu}$ centered at $2^{-\ga}\vec\mu$ with length
$2^{1-\ga}$.  For a sequence $\eta=\{\eta^{\ga,G}_{\vec\mu}\}$ define the norm
$$
\|\eta|f^s_{r,q}\|=
\Big\|(\sum_{\ga,G,\vec\mu}2^{\ga sq}|\eta^{\ga,G}_{\vec\mu}\chi_{\ga\vec\mu}(\cdot)|^q)^{1/q}\Big\|_{L^r(\mathbb R^{2n})}.
$$

Let $\mathbb N\ni k>\max
\{s,\f{4n}{\min(r,q)}+n-s\}$.
Let $\Psi_{\vec\mu}^{\ga,G}$ be the $2n$-dimensional Daubechies wavelet with smoothness $k$
as in Lemma \ref{wave}.
Let $m\in \mathcal S'(\mathbb R^{2n})$. Then $m\in F^s_{r,q}(\mathbb R^{2n})$ if and only if
it can be represented as 
$$
m=\sum_{\ga,G,\vec\mu}\eta^{\ga,G}_{\vec\mu}2^{-\ga n}\Psi^{\ga,G}_{\vec\mu}
$$
with $\|\eta|f_{rq}^s\|<\nf$ with 
unconditional convergence   in $\mathcal S'(\rn)$. Furthermore this representation
is unique,
$$
\eta_{\vec\mu}^{\ga,G}=2^{\ga n}\langle m,\Psi^{\ga,G}_{\vec\mu}\rangle,
$$
and
$$
I: m\to\big\{2^{\ga n}\langle m,\Psi^{\ga,G}_{\vec\mu}\rangle\big\}
$$
is an isomorphism from $F^s_{r,q}(\mathbb R^{2n})$
onto $f^s_{r,q}.$

\end{lm}

We now return the multipliers  considering their wavelet decompositions.
Before doing so,  we make some comments. The functions 
$\psi_F$ and $\psi_M$ have compact supports, and all elements in
a fixed level, i.e., of the same dilation factor $\ga$,
in the basis come from translations of finitely many products, so their supports have finite  overlaps.
Consequently we can  
classify the elements in the basis into finitely many classes so that
all elements in the same level in each class have distant supports, which means
that if $\om$ and $\om'$ are in the same class with the same dilation parameter $\ga$,
then $5\text{ supp }\om\cap 5\text{ supp }\om'=\emptyset$,
where $5\text{ supp }\om=B(c_0,5d)$ with $c_0$ inside the support of $\om$ and $d$ the diameter of the support of $\om$.
So, from now on, we will assume that the supports of $\om$'s 
related to a given dilation
factor $\ga$ are far disjoint.

For the multiplier $M_j$ in Theorem~\ref{04141}, we have a wavelet decomposition
using Lemma \ref{08311}, i.e.
\begin{equation}\label{emj}%equation of M_j
M_j=\sum a_{\om}\om,
\end{equation}
where the summation is over all $\om=\Psi^{\ga,G}_{\vec \mu}$ in the orthonormal basis 
described in Lemma \ref{wave}, the order of cancellations of $\psi_M$ is $M=4n+6$, and
$a_\om=\langle M_j,\om\rangle$.

Concerning the size of $a_\om=a_{k,l}$, we have the following estimate.

%which is a direct implication of Lemma \ref{smooth} and Corollary \ref{mm}.

\begin{cor}\label{06061}
The coefficient $a_\om$ in \eqref{emj} related to $\om$ with dilation $\ga$ is bounded by 
$C2^{-j\la}2^{-(s+n-\tf{2n}{r})\ga}$.

\end{cor}

\bpf

Since $F^s_{r,2}(\mathbb R^{2n})=L^r_s(\bbr^{2n})$,
we have 
$$
\bigg\|\Big(\sum_{k,l}2^{2\ga s}|2^{\ga n}a_{k,l} 
\chi_{Q_{\ga, k,l}}|^2\Big)^{1/2}\bigg\|_{L^r} \leq C\|M_j\|_{L^r_s},
$$
by Lemma~\ref{08311},
where $Q_{\ga, k,l}$ is the cube  centered at $2^{-\ga}(k,l)$ with length
$2^{1-\ga}$. 
Take  just one term on the left hand side,
and notice that $|Q_{\ga, k,l}|\sim 2^{-2n\ga}$,
then 
$$
|a_{k,l}|\le CA 2^{-j\la}2^{-(s+n)\ga}2^{2n\ga/r}\le C A 2^{-j\la}2^{-\ga(s+n-\tf{2n}{r})}.
$$

\epf

With the wavelet decompositions in hand, we are able to prove Theorem~\ref{04141}. The proof
is inspired by \cite{Grafakos2016a} and the square function technique (see  \cite{Carbery1983} and  \cite{RubiodeFrancia1986}). We control 
\begin{equation}\label{e12091}
T_j=\sup_{t>0}\bigg|\iint_{\bbr^{2n}}M_j(t\xi,t\eta)\wh f(\xi)\wh g(\eta)e^{2\pi ix\cdot(\xi+\eta)}
d\xi d\eta\bigg|
\end{equation}
by two integrals with the diagonal  and the off-diagonal parts. For the diagonal part we have just one term, which can be
handled using product wavelets. For the off-diagonal parts
 we introduce two square operators with each one 
 bounded by a product of the Hardy-Littlewood maximal function and a  linear operator bounded on $L^2(\rn)$. 

We need to decompose $M_j$ further. Take $N$ to be a fixed large enough number 
so that $N/10$ is greater than $d$, the diameters of the support of  $\om$ with dilation factor $\ga=0$.
We write $\om(\xi,\eta)=\om_{\vec \mu}(\xi,\eta)=\om_{1,k}(\xi)\om_{2,l}(\eta)$,
where $\vec\mu=(k,l)$ with $k,l\in \mathbb Z^n$, and denote the corresponding 
coefficient $\langle\om_{k,l},M_j\rangle$ by $a_{k,l}$. We 
define
\begin{equation}\label{e02131}
M_j^1=\sum_{\ga\ge 0}M_{j,\ga}^1=\sum_{\ga}\sum_{|k|\ge N} \sum_{|l|\ge N} a_{k,l}\om_{1,k}\om_{2,l}%=
%M_j^1+M^2_j.
\end{equation}
\begin{equation}\label{DefMj2}
M_j^2=\sum_{\ga\ge 0}M_{j,\ga}^2=\sum_{\ga}\sum_{k} \sum_{|l|\le N} a_{k,l}\om_{1,k}\om_{2,l}%=
%M_j^1+M^2_j.
\end{equation}
\begin{equation}
M_j^3=\sum_{\ga\ge 0}M_{j,\ga}^3=\sum_{\ga}\sum_{|k|\le N} \sum_{|l|\ge N} a_{k,l}\om_{1,k}\om_{2,l}.%=
%M_j^1+M^2_j.
\end{equation}
%(\sum_{|l|\ge N}+
Here $M_j^1$ is the diagonal part such that the support of each level is away from both $\xi$ and $\eta$ axes,
$M_j^2$ is the off-diagonal part with each level's support near the $\xi$ axis, and the support of  each level of
$M_j^3$
 is near the $\eta$ axis. 
 
% \bq The relation between the off-diagonal part and the paraproducts?
 
 %\bq * What do we have if we select a larger off-diagonal part? More precisely, if the diagonal part contains fewer terms,
% do we have better decay for this part?
% Choose $N$ depending on $j$ and $\ga$.

 \begin{rmk}
This decomposition is more delicate than that in \cite{Barrionuevo2017}. 
and allows us to   handle
more singular operators.
Actually for each fix $\ga$, the supports of the wavelets in $M_j^2$ related to $\ga$ are contained in
 $\{(\xi,\eta):\, |\eta|\le Nd2^{-\ga}\}$, while the corresponding part in \cite{Barrionuevo2017} is contained in
$\{(\xi,\eta):\, |\eta|\le 2^{j\ep}\}$.
  \end{rmk}
 
 Corresponding to $M_{j,\ga}^i$, $i=1,2,3$, we  define
% $m^i_j(\xi,\eta)=M_j^i(2^j\xi, 2^j\eta)$ for $i=1,2,3$.
% Denote   
$$
B^i_{j,\ga,t}(f,g)(x)=
\iint_{\mathbb R^{2n}}M^i_{j,\ga}(t\xi,t\eta)\wh f(\xi)\wh g(\eta)e^{2\pi i x\cdot (\xi+\eta)}d\xi d\eta.
$$
We can define $B_{j,\ga,t}$ in a similar way and   $B_{j,\ga,t}(f,g)(x)=\sum_{i=1}^3B_{j,\ga,t}^i(f,g)(x)$.
Moreover we can define $T_{j,\ga}^i$ in the way similar to \eqref{e12091} so that
$T_j=\sum_{i=1}^3\sum_\ga T_{j,\ga}^i$.

%\newpage

\section{Proof of Theorem~\ref{04141}}\label{06091}

%\bq Track $r=1$ and $r=4$.

\begin{comment}

For a wavelet $\om=\Psi^{\la,G}_{\vec \mu}$ such that
$a_\om=\langle\om,m_j\rangle\neq 0$, it is in  one of the following
 two classes. The first class consists of 
 wavelets supported in the diagonal set $\{(\xi,\eta): 2^{-Mj}|\eta|\le |\xi|\le 2^{Mj}|\eta|\}$ and 
 the second class consists of wavelets whose supports intersect the off-diagonal set
 $\{(\xi,\eta): 2^{Mj}|\eta|\le |\xi| \textrm{ or } 2^{Mj}|\xi|\le |\eta|\}$
 separately. We notice that 
 $|a_\om|\le  C C_M2^{-(M+1+n)\ga}$ by Lemma \ref{smooth}

For a wavelet $\om=\Psi^{\la,G}_{\vec \mu}$ such that
$a_\om=\langle\om,m^1_j\rangle\neq 0$, it is % in  one of the following
% two classes. The first class consists of wavelets 
supported in the diagonal set $\{(\xi,\eta): 2^{-Mj}|\eta|\le |\xi|\le 2^{Mj}|\eta|\}$.%, where $M=4n $. 
%  and  the second class consists of wavelets whose supports intersect the off-diagonal set $\{(\xi,\eta): 2^{Mj}|\eta|\le |\xi| \textrm{ or } 2^{Mj}|\xi|\le |\eta|\}$ separately.

 %\subsection{The Diagonal Part}
 
 %We consider the wavelets supported in the set $\{(\xi,\eta): 2^{-Mj}|\eta|\le |\xi|\le 2^{Mj}|\eta|\}$ and  denote the corresponding multiplier by $m_j^1=\sum_\om a_\om \om$,  where the sum is over all wavelets supported in $\{(\xi,\eta): 2^{-Mj}|\eta|\le |\xi|\le 2^{Mj}|\eta|\}$.

 \end{comment}

 For $f,g\in \mathcal S(\rn)$, %using % the relations between $M_j$ and $m_j$ by their definitions and  Corollary \ref{mm}, 
using the fundamental theorem of Calculus,  we   rewrite
 \begin{align*}
 &B^1_{j,\ga,t}(f,g)(x)\\
 =&\iint_{\mathbb R^{2n}}  M_{j,\ga}^1(t\xi,t\eta)\wh f(\xi)\wh g(\eta)
 e^{2\pi ix\cdot (\xi+\eta)}d\xi d\eta\\
 =&\int_0^t\iint_{\mathbb R^{2n}}   (s\xi,s\eta)\cdot\nabla M_{j,\ga}^1(s\xi,s\eta)\wh f(\xi)\wh g(\eta)
 e^{2\pi ix\cdot (\xi+\eta)}d\xi d\eta\f{ds}s , 
 %=&\int_0^t\iint (2^js\xi,2^js\eta)\cdot\nabla M_j^1(2^js\xi,2^js\eta)\wh f(\xi)\wh g(\eta)
 %e^{2\pi ix\cdot (\xi+\eta)}d\xi d\eta\f{ds}s,
 \end{align*}
 where the existence of $\nabla M_{j,\ga}^1$
% and $\nabla m_j^1$ are 
is guaranteed by that all components in $M_{j,\ga}^1$ are contained in the same level.

 Define the operator  related to $( s\xi, s\eta)\cdot\nabla M_{j,\ga}^1( s\xi, s\eta)$
 as
 $$
\tilde B_{j,\ga,s}^1(f,g)(x)=\iint_{\mathbb R^{2n}} 
 ( s\xi, s\eta)\cdot\nabla M_{j,\ga}^1( s\xi, s\eta)\wh f(\xi)\wh g(\eta)
 e^{2\pi ix\cdot (\xi+\eta)}d\xi d\eta.
 $$ 
 Then we have the pointwise estimate
 \begin{equation}\label{Tj1}
 T^1_{j,\ga}(f,g)(x)=\sup_{t>0} |B^1_{j,\ga,t}(f,g)(x)|\le \int_0^{\nf}|\widetilde B_{j,\ga,t}^1(f,g)(x)|\f{dt}t
 \end{equation}
 
We now turn to the study of the boundedness of  $\widetilde B_{j,\ga,t}^1$.
The basic idea is the observation that when $r\in(1,4)$, \cite[Remark 2]{Grafakos2016a}
shows that whenever $\si$ is supported in $B(0,R)$ we have
$$
\|T_\si(f,g)\|_{L^2\times L^2\to L^1}
\le C \|\si\|_{L^r_s}
$$ with $C$ independent of $R$.

To make this argument rigorous, %for a fixed $\ga$ we set 
%$$\tilde B_{j,t,\ga}^1(f,g)(x)=\iint ( t\xi, t\eta)\cdot\nabla M_{j,\ga}^1( t\xi, t\eta)\wh f(\xi)\wh g(\eta)
% e^{2\pi ix\cdot (\xi+\eta)}d\xi d\eta,$$
% where
%$\nabla M_{j,\ga}^1(\xi,\eta)=\sum_k\sum_la_{k,l} 
%\nabla_{(\xi,\eta)}(\om_{1,k}\otimes\om_{2,l})(\xi,\eta)
%$
%with $\om=\om_{1,k}(\xi)\om_{2,l}(\eta)$ related to the dilation factor
%$\ga$.
\begin{comment}
 Let $A^1_{j,t}(f,g)(x)=\iint m_j^1(t\xi,t\eta)\wh f(\xi)\wh g(\eta)
 e^{2\pi ix\cdot (\xi+\eta)}d\xi d\eta$.
Then we have $A^1_{j,t}(f,g)(x)=\int_0^ts\f{dA^1_{j,s}(f,g)}{ds}\f{ds}{d}\le\int_0^{\nf}|\widetilde A^1_{j,s}(f,g)(x)|\f{ds}{s}$, where $\widetilde A^1_{j,s}$ related to the multiplier 
$\widetilde m_j^1=(\xi,\eta)\cdot(\nabla m_j^1)(\xi,\eta)$ is defined as
$$\widetilde A^1_{j,s}(f,g)(x)=\iint (s\xi,s\eta)\cdot(\nabla m_j^1)(s\xi,s\eta)\wh f(\xi)\wh g(\eta)
 e^{2\pi ix\cdot (\xi+\eta)}d\xi d\eta$$
 
\bq Notice that $|\p^{\al}m^1_j|\le C2^{-j(\la-|\al|)}$ and  
$|\p^{\al}(\widetilde m^1_j)|\le C2^{-j(\la-1-|\al|)}$. %We have also that$\|\wtd m^1_j\|_2\le C 2^{-j(\la-1)-\f j2}= C2^{-j(\la-\f12)}.$
{\color{red}
Let us prove the claim for $m_j^1$ first. 
}
\end{comment}
in the  case $t=1$, we have the following estimate, whose proof 
can be found in the Appendix (Section~5).

\begin{prop}\label{BL}%B one level
%Let $s>\tf n2$, then f
%For any $\ep>0$
%  $\tilde B^1_{j,1,\ga}$ 
Let $E=\{\xi\in \rn:\,\, C2^{-\ga}\le|\xi|\le 2^j\}$.
Then we have  
\begin{equation}\label{e06061}
\|\widetilde B^1_{j, \ga,1}(f,g)\|_{L^1}\le
C  A C(j,\ga)\|\wh f\chi_{E}\|_{L^2}
\|\wh g\chi_{E}\|_{L^2},
\end{equation}
with 
\begin{eqnarray*}
C(j,\ga) &= &n(j+\ga)2^{-j(\la-1)}2^{\ga(1+\tf n2-s)}   \quad   \textup{when $r=4$} \\
 C(j,\ga) & =& 2^{-j(\la-1)}2^{\ga(1+\tf {2n}r-s)} \qquad\quad\quad   \textup{when $1<r<4$.} 
\end{eqnarray*}
 \end{prop}

In both cases we have good decay in $j$. 

%\bq Do we need $s>2n+1$ when $r=1$?

%In the proof of \cite[Lemma 6]{GHH}, we use mainly (i) the size of $a_{k,l}$, (ii) the disjointness of the supports of $\om$'s, and (iii) that $\om$ is a tensor product. 
%Notice that we will have enough decay in $j$ and $\ga$ if
%$\la>(4n+3)/5$ and $M>4n+5$, which are satisfied by our assumptions on $\la$ and
%choices of wavelets.
\begin{cor}\label{dj}%diagonal j
For the diagonal part we have
$$
\|T_{j,\ga}^1(f,g)\|_{L^1}\le C A C(j,\ga)(j+\ga)\|f\|_{L^2}\|g\|_{L^2}.
$$
%where $ C(j,\ga)(j+\ga)$.% is as in Proposition \ref{BL}.
\end{cor}

\begin{proof}
\begin{comment}
Then by $\|\wtd m^1_j\|_2\le C 2^{-j(\la-1)-\f j2}= C2^{-j(\la-\f12)},$
and Lemma \ref{wavesize}
% Corollary 8 in \cite{GraHeHon} 
we have that
$$\|\wtd A_{j,1}(f,g)\|_{L^1}\le C 2^{-j(\la-1-M)/5} 2^{-4j(\la-\f12)/5}= C 2^{-j(\la-\f{M+3}5)}.$$

%Let $\wh {f^j}= \widehat{f}\chi_{\{2^{-Mj}\leq |\xi| \leq 1\}}$.
%By our choice of $\wtd m^1_j$, w
We know that
$\wtd A^1_{j,1}(f,g)=\wtd A^1_{j,1}(f^j,g^j)$ 
 because of the support of
$\wtd m^1_j$, where 
$\wh {f^j}= \widehat{f}\chi_{E_{j,1}}$ and
$E_{j,t}=\{\xi\in\mathbb R^n: \f{2^{-Mj}}t\leq |\xi| \leq \f{1}t\}$.
Notice that
$\wtd A^1_{j,s}(f,g)(x)=s^{-2n}\wtd A^1_{j,1}(f_s,g_s)(\f xs)$ with %, where 
$\wh{f_s}(\xi)=\wh f(\xi/s)$.
Then 
\begin{align*}
\|\wtd A_{j,s}^1(f,g)\|_{L^1}\,\,\le &%\,\, C2^{-j\de/5}t^{-n}\|\wh f(\xi/t)\chi_{E_{j,0}}\|_{L^2}\|\wh g(\eta/t)\chi_{E_{j,0}}\|_{L^2}\\
%\le &
\,\, C2^{-j(\la-\f{M+3}5)} \|\wh f\chi_{E_{j,s}}\|_{L^2}\|\wh g\chi_{E_{j,s}}\|_{L^2}.
\end{align*}

Then for $T^1_j(f,g)(x)=\sup_{t>0}|A^1_{j,t}(f,g)(x)|$
%Putting all previous argument together, 
we can control $
\|T_j^1(f,g)\|_{L^1}$
by
\end{comment}
From \eqref{Tj1} we know that
$$\|T_{j,\ga}^1(f,g)\|_{L^1}\le
\int_0^{\nf}\|\widetilde B^1_{j,\ga,t}(f,g)\|_{L^1}\f{dt}{t}
=
\int_0^{\nf}\|\widetilde B^1_{j, \ga,1}(f_t,g_t)\|_{L^1}\f{dt}{t},
$$
where $\wh f_t(\xi)=t^{-n/2}\wh f(\xi/t)$, and $\wh g_t(\xi)=t^{-n/2}\wh g(\xi/t)$. 
Applying Proposition~\ref{BL}, the last integral is dominated by
\begin{align*}
& C  A 
\int_0^{\nf} C(j,\ga) \|\wh f_t\chi_{E}\|_{L^2}\|\wh g_t\chi_{E}\|_{L^2}\f{dt}t\\
\le &C  C(j,\ga)\Big(\int_{\rn}\int_0^\nf|\wh f_t\chi_{E}|^2\f{dt}td\xi\Big)^{1/2}
\Big(\int_{\rn}\int_0^\nf|\wh g_t\chi_{E}|^2\f{dt}td\xi\Big)^{1/2}.
\end{align*}
The double integral involving $f_t$ is bounded by
$\int|\wh f(\xi)|^2\int_{C2^{-\ga}/|\xi|}^{2^j/|\xi|}\tf{dt}td\xi$, which is less than
$C(j+\ga)\|f\|_{L^2}^2$.
Hence the last expression is controlled by
$C  A C(j,\ga)(j+\ga)\|f\|_{L^2}\|g\|_{L^2}$.
\end{proof}

\medskip

We next deal with  
the off-diagonal parts. More specifically, we  consider 
$B_{j,\ga,t}^2$, since the analysis of $B_{j,\ga,t}^3$ is similar in view of symmetry.
Recall that 
$$
B_{j,\ga,t}^2=\iint_{\mathbb R^{2n}}M^2_{j,\ga}(t\xi,t\eta)\wh f(\xi)\wh g(\eta)e^{2\pi i x\cdot (\xi+\eta)}d\xi d\eta.
$$
%We denote $m_j^2=\sum_\om a_\om \om$, wherethe sum is over all wavelets  supported in $\{(\xi,\eta): 2^{M(1-\ep)j}|\eta|\le |\xi|\}$.The remaining part supported in $\{(\xi,\eta): 2^{M(1-\ep)j}|\xi|\le |\eta|\}$ could beproved similarly.
We denote $(\xi,\eta)\cdot(\nabla M_{j,\ga}^2)(\xi,\eta)$ by 
$\widetilde M_{j,\ga}^2(\xi,\eta)$.
Then similar to $B^2_{j,t}(f,g)(x)$ we define %the operator
$$
\widetilde B^2_{j,\ga,t}(f,g)(x)
  =\iint_{\mathbb R^{2n}}  \wtd M_{j,\ga}^2(t\xi,t\eta)\wh f(\xi)\wh g(\eta)
 e^{2\pi ix\cdot (\xi+\eta)}d\xi d\eta .
 $$
% Like before we can define $B^2_{j,s}$ and $\tilde B^2_{j,s}$ similar to 
% $A^2_{j,s}$ and $\tilde A^2_{j,s}$ with the appearances of $m$ replace by $M$.
% A simple calculation shows that $A^2_{j,s}=B^2_{j,2^js}$, and $\tilde A^2_{j,s}=\tilde B^2_{j,2^js}$.
With   these notations, by the fundamental theorem of Calculus, we have
\begin{align*}
(B^2_{j,\ga,t}(f,g)(x))^2=\,\,&\,\, 2\int_0^tB^2_{j,\ga,s}(f,g)(x)s\f{dB^2_{j,\ga,s}(f,g)(x)}{ds}\f{ds}{s}\\
%=&2\int_0^tB^2_{j,2^js}(f,g)(x)\tilde B^2_{j,2^js}(f,g)(x)\f{ds}{s}\\
\le \,\,&\,\, 2\int_0^{\nf}|B^2_{j,\ga,s}(f,g)(x)||\widetilde B^2_{j,\ga,s}(f,g)(x)|\f{ds}{s}\\
\le \,\,&\,\, 2 G_{j,\ga}(f,g)(x)\wtd G_{j,\ga}(f,g)(x),
\end{align*}
where we set 
\begin{align*}
G_{j,\ga}(f,g)(x)& =\bigg(\int_0^{\nf}|B^2_{j,\ga,s}(f,g)(x)|^2\f{ds}{s}\bigg)^{1/2} \\
\wtd G_{j,\ga}(f,g)(x)& =     \bigg(\int_0^{\nf}|\widetilde B^2_{j,\ga,s}(f,g)(x)|^2\f{ds}{s}\bigg)^{1/2}.
\end{align*}  
These 
$g$-functions are bounded from $L^2\times L^2$ to $L^1$ with good decay in $j$.  Indeed, we have the following.

\begin{lm}\label{gfn}
For any $\ep>0$
there exists a constant $C_\ep$ independent of $j$ such that for all $f,g\in \mathcal S(\rn)$,
$$
\|G_{j,\ga}(f,g)\|_{L^1}\le C_\ep A 2^{-j\la}2^{(n+1)\ga}\|f\|_{L^2}\|g\|_{L^2}
$$
and
$$
\|\widetilde G_{j,\ga}(f,g)\|_{L^1}\le C_\ep A 2^{-j(\la-1)}2^{(n+1)\ga}\|f\|_{L^2}\|g\|_{L^2}.
$$
\end{lm}

%We delete some details which can be found in %t
The proof of this lemma is inspired by  
\cite{Grafakos2015}.

\begin{proof}
We will focus on $\wtd G_{j,\ga}$ first. % because the proof of $ G_{j,\ga}$ is similar and simpler. 
For $\wtd G_{j,\ga}$ we need to consider two typical cases, the derivative falling on $\xi$
and the derivative falling on $\eta$.

Let us consider the multiplier 
$$
\xi_1\p_{\xi_1}M_{j,\ga}^2 = \sum_{k,l}
a_{k,l}v_{k}(\xi)\om_{2,l}(\eta)
$$
 with 
$v_{k}(\xi)=\xi_1\p_{\xi_1}\om_{1,k}(\xi)$. 
Using \eqref{defomega1k} we observe that 
$$
|v_k(\xi)|\le C 2^j2^{\ga n/2}2^{\ga}=C2^j2^{\ga(n+2)/2}.
$$
The $g$-function related to $\xi_1\p_{\xi_1}M_{j,\ga}^2$ is denoted by $\wtd G_{j,\ga}^1(f,g)$.
By the definition \eqref{DefMj2}, 
for a  fixed $\ga$  at most $N$  of $\om_{2,l}$ are involved, so we can consider  a single fixed $l$.
Observe that
\begin{align*}
&\iint_{\bbr^{2n}}\sum_ka_{k,l}v_k(\xi)\om_{2,l}(\eta)\wh f(\xi)\wh g(\eta)e^{2\pi ix\cdot (\xi,\eta)}
d\xi d\eta\\
=&\|a\|_{\ell^\nf}2^{\ga(n+2)/2}2^j \bigg(\int_{\rn}\om_{2,l}(\eta)\wh g(\eta)e^{2\pi ix\cdot \eta}d\eta
\bigg) \\
& \qquad\qquad\qquad\qquad\qquad
\int_{\rn}\f{\sum_ka_{k,l}v_k(\xi)}{\|a\|_{\nf}2^{\ga(n+2)/2}2^j}\wh f(\xi)e^{2\pi ix\cdot\xi}d\xi.
\end{align*}
By $|v_k|\le C2^{\ga(n+2)/2}2^j$, 
$|a_{k,l}|\le \|a\|_{\ell^\nf}$, and  the disjointness of the supports of $v_k$,
we know that
$\si(\xi):=(\sum_ka_{k,l}v_k(\xi))/(\|a\|_{\ell^\nf}2^{\ga(n+2)/2}2^j)$ is a 
compactly supported bounded function. Hence the bilinear operator related to the multiplier $\sum_ka_{k,l}v_k(\xi)\om_{2,l}(\eta)$ is  pointwise bounded by 
\begin{equation}\label{poinwest}
C\|a\|_{\ell^\nf}2^{\ga(n+1)}2^jM(g)(x)T_{\si}(f)(x),
\end{equation}
where
$T_{\si}(f)$ satisfies that
$\|T_{\si}(f)\|_{L^2}\le C\|\wh f\chi_{F}\|_{L^2}$ with 
$$
F=\{\xi\in \mathbb R^n: 2^{j-1}\le |\xi|\le 2^{j+1}\}.
$$

The operator $\wtd G_{j,\ga}^1$ is then bounded from $L^2\times L^2$ to $L^1$. Indeed we can estimate it by 
a standard dilation argument as follows.
Setting $\wh f_s(\xi)=s^{-n/2}\wh f(\xi/s)$, and $\wh g_s(\xi)=s^{-n/2}\wh g(\xi/s)$, we have
%$|A_{j,s}^2(f,g)(x)|\le 2^{-j(\la-M\ep)}s^{-n}M(g)(x)T_\si(f_s)(x/s), $which implies that
 \begin{align*}
 &\int_{\mathbb R^n}\wtd G_{j,\ga}^1(f,g)(x)dx\\
 =&\int_{\mathbb R^n}\bigg[\int_0^{\nf}\Big|
 \iint_{\bbr^{2n}} \sum_{k,l}
a_{k,l}v_{k}(s\xi)\om_{2,l}(s\eta)\wh f(\xi)\wh g(\eta)e^{2\pi ix\cdot (\xi,\eta)}
d\xi d\eta\Big|^2\f{ds}s\bigg]^{\frac12}\! dx\\
=&\int_{\mathbb R^n} \!\! \bigg[ \! \int_0^{\nf}\Big|
 \iint_{\bbr^{2n}} \sum_{k,l}
a_{k,l}v_{k}(\xi)\om_{2,l}(\eta)\wh f_s(\xi)\wh g_s(\eta)e^{2\pi i \f xs\cdot (\xi,\eta)}
\f{ d\xi d\eta}{s^n}\Big|^2\f{ds}s\bigg]^{\frac12}\! dx. 
\end{align*}
Since  there are only finitely many $l$ in the sum above,  
we can use the pointwise estimate~\eqref{poinwest}
to estimate the last displayed expression  by 
\begin{align*}
  &C  \|a\|_{\ell^\nf}2^j2^{(n+1)\ga}\int_{\rn}
\Big(\int_0^\nf \Big|s^{-n/2}M(g)(x)T_{\si}(f_s)(s^{-1}x) \Big|^2\f{ds}{s}\Big)^{1/2}dx\\
\le &C  \|a\|_{\ell^\nf}2^j2^{(n+1)\ga}\|M(g)\|_{L^2}\bigg(\int_0^{\nf}\int_{\mathbb R^n}s^{-n}
|\wh f(\xi/s) |^2\chi_{F}(\xi) d\xi\frac{ds}{s} \bigg)^{\frac12}\\
\le &  C \|a\|_{\ell^\nf}2^j2^{(n+1)\ga}\|g\|_{L^2}\Big(\int_{\mathbb R^n}|\wh f(\xi)|^2
\int_{(2^{j-1})/|\xi|}^{(2^{j+1})/|\xi|}\f{ds}sd\xi\Big)^{\frac12}.
\end{align*}
%\le & C2^{-j(\la-M\ep)}\|g\|_{L^2}\|f\|_{L^2}.
The integral with respect to $s$ is $\log\f{2^{j+1}}{2^{j-1}}\le C$. This, combined
with the bound of $\|a\|_{\ell^\nf}\le C A 2^{-j\la}2^{-(s+n-\tf{2n}{r})\ga}$ obtained in Corollary~\ref{06061}, shows that the last displayed 
expression is smaller than
\begin{equation}\label{gj1}
C \|a\|_{\ell^\nf}2^j2^{(n+1)\ga}\|g\|_{L^2}\|f\|_{L^2}\le C A 2^{-j(\la-1)}2^{-\ga (s-\tf {2n}r-1)}\|g\|_{L^2}\|f\|_{L^2}.
\end{equation}
%since $\sum_\ga 2^{-\ga (s+n-\tf {2n}r-n-1)}=\sum_\ga 2^{-\ga (s-\tf {2n}r-1)}<\nf$.

When the derivative falls on $\eta$, for example we have differentiation with respect to $\eta_1$, 
using the notation $v_{l}(\eta)=\eta_1\p_{\eta_1}\om_{2,l}(\eta)$
we have a similar representation
\begin{align*}
&\iint_{\bbr^{2n}}\sum_ka_{k,l}\om_{1,k}(\xi)v_l(\eta)\wh f(\xi)\wh g(\eta)e^{2\pi ix\cdot (\xi,\eta)}
d\xi d\eta\\
=&\|a\|_{\ell^\nf}2^{\ga n/2} \bigg(\int_{\rn}v_l(\eta)\wh g(\eta)e^{2\pi ix\cdot \eta}d\eta\bigg)
\int_{\rn}\f{\sum_ka_{k,l}\om_{1,k}(\xi)}{\|a\|_{\ell^\nf}2^{\ga n/2}}\wh f(\xi)e^{2\pi ix\cdot\xi}d\xi.
\end{align*}
The   integral in the parenthesis in the last line is  dominated by
$2^{\ga n/2}M(g)(x)$   as  both
$\p_1(\om_{2,l})^{\vee}(x)e^{2\pi ix\cdot l}$ and $(\om_{2,l})^{\vee}(x)l_1e^{2\pi ix\cdot l}$
are Schwartz functions, and the number of the second type of functions is finite because 
$|l|\le N$. The bilinear operator related to the multiplier
$\sum_ka_{k,l}\om_{1,k}(\xi)v_l(\eta)$
is therefore bounded by  
$$
C\|a\|_{\ell^\nf}2^{\ga n}M(g)(x)T_{\si'}(f)(x) ,
$$ 
where $T_{\si'}$ satisfies the same property as
$T_\si$.
For the $L^1$ norm of the $g$-function $\wtd G^2_{j,\ga}$ related to the multiplier $\sum_ka_{k,l}\om_{1,k}(\xi)v_l(\eta)$
we apply an argument similar to that used for $\|\wtd G^1_{j,\ga}\|_{L^1}$. We obtain 
\begin{equation}\label{gj2}
\|\wtd G^2_{j,\gamma}\|_{L^1}\le 
C A 2^{-j\la}2^{-\ga (s-\tf {2n}r)}
\|g\|_{L^2}\|f\|_{L^2} .
%\le  C2^{-j\la}2^{(n+1)\ga}\|f\|_{L^2}\|g\|_{L^2}
\end{equation}
This estimate and
\eqref{gj1}
show that 
$$
\|\wtd G_{j,\ga}(f,g)\|_{L^1}\le C A 2^{-j(\la-1)}2^{-\ga (s-\tf {2n}r-1)}\|f\|_{L^2}\|g\|_{L^2}.
$$

For $G_{j,\ga}(f,g)$ an analogous, but simpler argument, applied to the standard representation 
$\sum a_{k,l}\om_{1,k}(\xi)\om_{2,l}(\eta)$ yields
$$
\|G_{j,\ga}(f,g)\|_{L^1}\le C A 2^{-j\la}2^{-\ga (s-\tf {2n}r-1)}\|f\|_{L^2}\|g\|_{L^2}.
$$
The additional decay of $2^{-j}$ comes from the fact that in the multiplier of $B_{j,\ga,s}^2$  we 
miss the term $(\xi,\eta)$, which is  controlled by $2^j$.
\end{proof}

%\bq  * Check the decay we obtained for the off-diagonal part.

\begin{cor}\label{odj}%off-diagonal j
For the off-diagonal part the estimate below holds:
\begin{equation}\label{e12093}
\|T_{j,\ga}^2(f,g)\|_{L^1}\le C A 2^{-j(\la-1/2)}2^{-\ga (s-\tf {2n}r-1)}\|f\|_{L^2}\|g\|_{L^2}.
\end{equation}
\end{cor}
\begin{proof}
By the calculation before Lemma \ref{gfn} we have the pointwise control
$$
T^2_{j,\gamma}(f,g)(x)\le \sqrt2 (G_{j,\ga}(f,g)(x)\wtd G_{j,\ga}(f,g)(x))^{1/2} , 
$$
which, combined with Lemma 
\ref{gfn}, implies that
\begin{align*}
\|T^2_j(f,g)\|_{L^1}\le &\big\|\sqrt2 \big( G_{j,\ga}(f,g) \wtd  G_{j,\ga}(f,g) \big)^{1/2}\big\|_{L^1}\\
\le &C(\| G_{j,\ga}(f,g) \|_{L^1}\|\wtd  G_{j,\ga}(f,g) \|_{L^1})^{1/2}\\
\le & C A \Big(2^{-j\la}2^{-j(\la-1)}2^{-2\ga (s-\tf {2n}r-1)}\|f\|_{L^2}^2\|g\|_{L^2}^2\Big)^{1/2}\\
= & C A 2^{-j(\la-1/2)}2^{-\ga (s-\tf {2n}r-1)}\|f\|_{L^2}\|g\|_{L^2}.
\end{align*}
In this case we have nice decay in $j$ for $T_{j,\ga}^2$ since %in the diagonal part we have the requirement
$\la>1>1/2$. 
\end{proof}

We collect the known results to finish the proof of Theorem~\ref{04141}.

\begin{proof}[Proof of Theorem~\ref{04141}]
We observe that 
$$
T(f,g)(x)\le\sum_{j=0}^\nf \sum_\ga |T_{j,\ga}(f,g)(x)|.
$$
It is straightforward to verify that
$$
\sum_{\ga}C(j,\ga)(j+\ga)\le C_\ep2^{-j(\la-1-\ep)}
\le C_\ep 2^{-j(\la-1)/2},
$$
if we choose $\ep$ small enough.
So we obtain
\begin{equation}\label{e12092}
\sum_j\sum_\ga\|T_{j,\ga}^1(f,g)\|_{L^1}\le \sum_jC A 2^{-j \f{\la-1}2}\|f\|_{L^2}\|g\|_{L^2}\le
C  A \|f\|_{L^2}\|g\|_{L^2}.
\end{equation}
This concludes the argument of the diagonal part. 
A similar argument using
\eqref{e12093} 
show the boundedness of the off-diagonal part.
Hence we deduce the conclusion of Theorem~\ref{04141}. 
\end{proof}

%\newpage

\section{Applications to bilinear maximal Bochner-Riesz}

Theorem~\ref{04141} can also be used to study the boundedness of the 
maximal bilinear  Bochner-Riesz means. These are the means
\begin{equation}\label{av}%Average
A^\la_{t}(f,g)(x)
= \iint_{\mathbb R^{2n}}\wh f(\xi)\wh g(\eta)\big(1-|t\xi |^2-|t \eta |^2\big)_+^{\la}e^{2\pi i x\cdot (\xi+\eta)}d\xi d\eta,
\end{equation}
which coincide with $B_{1/t}^\la(f\otimes g)(x,x)$ with $B^\la_{1/t}$ the linear 
Bochner-Riesz operator
on $\bbr^{2n}$ and $x\in\rn$. For test functions  we should have
$A^\la_t(f,g)\to fg$ as $t\to0$ in   
the $L^p$ or in the pointwise sense.   
\cite{Grafakos2006a, Bernicot2015a, Jeong2017}
have proved   positive results
 for $\la=0$ and $\la>0$ respectively, 
concerning their $L^p$ convergence. 

In this section, we are concerned with the pointwise 
convergence of the means \eqref{av}, in particular with the 
boundedness of the  maximal bilinear Bochner-Riesz 
operator, which of course implies the boundedness of the
bilinear Bochner-Riesz operators in the same range.

%\bq *Definition of bilinear maxiaml B-R.

The bilinear 
maximal Bochner-Riesz 
 operator  for $\la>0$
 is defined as
\begin{equation}\label{bBR}%bilinear Bochner-Riesz
T^{\la}_*(f,g)(x)=\sup_{t>0}
\Big|\int_{\mathbb R^n}\int_{\mathbb R^n}m^{\la}(t\xi,t\eta)\wh f(\xi)\wh g(\eta)e^{2\pi i x\cdot (\xi+\eta)}d\xi d\eta\Big|,
\end{equation}
where $ m^{\la}(\xi,\eta)=(1-|\xi |^2-|\eta |^2)_+^{\la}$,
which is equal to $(1-(|\xi |^2+|\eta |^2))^{\la}$ when $|(\xi,\eta)|\le 1$ and $0$ when
$|(\xi,\eta)|> 1$.

Our main theorem concerning the boundedness of bilinear maximal Bochner-Riesz means is as follows: 

 \begin{thm}\label{Main}
 When $\la> \tf{2n+3}4$, for $T^\la_*$ in \eqref{bBR}
 we have that
 $$
 \|T^\la_*(f,g)\|_{L^1}\le C\|f\|_{L^{2}}\|g\|_{L^{2}}.
 $$
% where $1<p_1,p_2<\nf$ and $1/p=1/p_1+1/p_2$.
 \end{thm}

We fix a nonnegative   smooth function $\varphi(s)$ supported in $[-\tf34,\tf34]$ 
and a smooth function $\psi$ supported in $[\tf18,\tf58]$ such that 
$ \sum_{j=0}^\nf\psi_j(1-s)=1$ for $s\in[0,1)$, where
$\psi_j(s)=\psi(2^js)$ for $j\ge 1$ and $\psi_0=\vp$.

%$\vp(s)=1$ for $s\le 1/2$ and $\vp(s)=0$ for $s\ge1$. Define 
%$\vp_j(s)=\vp(2^{j+1}(s+2^{-j}-1))$ for $j\ge 1$. 
%Denote by $\psi_j(s)$ the function $\vp_1(s)$ when $j=0$ and
%$\vp_{j+1}(s)-\vp_j(s)$ when $j\ge 1$. Notice that $\psi_0$ is supported in $(-\nf,3/4]$
%and $\psi_j$ is supported in $[1-2^{-j},1-2^{-j-2}]$ for $j\ge1$.
%Moreover $\sum_{j=0}^{\nf}\psi_j(s)=\chi_{(-\nf,1)}(s)$.

%Define $m(\xi,\eta)=(1-(|\xi|^2+|\eta|^2))_+^{\la}$, then
 We decompose the multiplier $m(\xi,\eta)=(1-(|\xi |^2+|\eta |^2))_+^{\la}$  smoothly as
  $m=\sum_{j\ge0}m_j$,
  where 
  $$
  m_j(\xi,\eta)=m(\xi,\eta)\psi_j(|(\xi,\eta)|) 
  $$
    is supported in an annulus of the form
 $$
 \{(\xi,\eta)\in\mathbb R^{2n}: {1-2^{-j}} \le |(\xi,\eta)| \le   {1-2^{-j-2}} \}
 $$
 for $j\ge 1$ and $m_0$ is supported in a ball of radius $3/4$ centered at the
 origin. 
 If 
\begin{equation}\label{e02261}
T_j(f,g)(x)=\sup_{t>0}
\bigg|\int_{\mathbb R^n}\int_{\mathbb R^n}\wh f(\xi)\wh g(\eta)m_j(t\xi,t\eta)e^{2\pi i x\cdot (\xi+\eta)}d\xi d\eta \bigg|,
\end{equation}
then 
$$
T^\la_*(f,g)(x)\le \sum_{j=0}^\nf T_j(f,g)(x).
$$
The following are straighforward facts about $T^\la_*$ and $T_j$. Let $\|T\|_{X\times Y\to Z}$
denote the norm of $T$ from $X\times Y$ to $Z$.
\begin{prop}\label{Uni} %Uniform bound
Assume $1<p_1,p_2\le\nf$, and 
$1/p=1/p_1+1/p_2$. Then for $\la>n-1/2$,
there exists a finite constant $C=C(p_1,p_2)$
such that
$\|T_*\|_{L^{p_1}\times L^{p_2}\to L^p}\le C$.
For any fixed $j$, there exists a finite constant $C_j(p_1,p_2)$
such that
$\|T_j\|_{L^{p_1}\times L^{p_2}\to L^p}\le C_j(p_1,p_2)$.

\end{prop}
\begin{proof}
Let us consider the kernel 
$K(y,z)=m^{\vee}(y,z)$ of $A_1^\lambda$ defined in \eqref{av},  which satisfies that
$|K(y,z)|\le C(1+|y|+|z|)^{-(n+\la+1/2)}$ (see, for example, \cite{Grafakos2014b}),
hence for $\la>n-1/2$, we have
\begin{align*}
 |A_t(f,g)(x)| 
=\,\,&\,\,\bigg|\int_{\mathbb R^{2n}}t^{-2n}K(\f{x-y}t,\f{x-z}t)f(y)g(z)dydz\bigg|\\
\le \,\,&\,\,C(\vp_t*|f|)(x)(\vp_t*|g|)(x)\\
\le \,\,&\,\,CM(f)(x)M(g)(x),
\end{align*}
where $M$ is the Hardy-Littlewood maximal function, and $\vp_t(y)=t^{-n}\vp(y/t)$
with $\vp(y)=(1+|y|)^{-(n+\la+1/2)/2}$,
which is integrable when $\la >n-1/2$.
Then $T_*(f,g)(x)\le CM(f)(x)M(g)(x)$,
which implies that $\|T_*(f,g)\|_{L^p}\le C(p_1,p_2)\|f\|_{L^{p_1}}\|g\|_{L^{p_2}}$ for
$1<p_1,p_2\le\nf$ with
$1/p=1/p_1+1/p_2$
in view of the boundedness of the Hardy-Littlewood maximal function.

We observe that each $m_j$ is smooth and compactly supported, hence for each  
$j$ a similar argument yields $\|T_j\|_{L^{p_1}\times L^{p_2}\to L^p}\le C_j(p_1,p_2)<\infty$.
 \end{proof}

 %whose proof we postpone
% to  next section, so what remains is  
%We next show that $T^\la_*$ is bounded from $L^{2}\times L^{2}$ to $L^1$ when
%$\la> \tf{2n+3}4$,
%which will be the main ingredient of the rest of this section.

% \section{A Smooth Decomposition}
 
% We consider the bilinear maximal operator $T_*(f,g)(x)=\sup_{t>0}|\int_{\mathbb R^n}\int_{\mathbb R^n}\wh f(\xi)\wh g(\eta)(1-\f{|\xi|^2-|\eta|^2}{t^2})_+^{\la}e^{2\pi i x\cdot (\xi+\eta)}d\xi d\eta|$. 
%\section{The Decompsitions}

%\bq Will we get something better if we take the decomposition (for both the space and frequency) in \cite{Carbery1988}?

%\section{The boundedness of $\|T_*\|$ at $(2,2,1)$ }
With the aid of the preceding decomposition and the boundedness of
$T_j$, the study of the boundedness of
$T_*$ is reduced to   the decay of $C_j$ in $j$.
%; for the case $(p_1,p_2)=(2,2)$, this is contained in the following proposition.\begin{prop}\label{De}% Decay
% For any $\ep>0$, we may find a constant $C=C_\ep$ such that
% $T_j$
% satisfies 
% $$\|T_j\|_{L^2\times L^2\to L^1}\le Cj2^{-j(\la-1-\ep)}.$$
 
 %\end{prop}

We now go back to the multipliers and will apply  Theorem~\ref{04141}.
For this purpose %, and to apply Lemma \ref{wave} and Lemma \ref{smooth},
we should study kinds of norms of $m_j$.

\begin{lm}\label{om}%Original Multiplier
There exists a constant $C$ such that
$$
\|m_j\|_{L^2}\le C2^{-j(\la+1/2)},
$$
and for any multiindex $\al$,
\begin{equation}\label{ofkkri}
\|\p^{\al}m_j\|_{L^{\nf}}\le C_\al2^{-j(\la-|\al|)}.
\end{equation}
\end{lm}

\begin{proof}
A change of variables using polar coordinates implies that
\begin{align*}
\|m_j\|_{L^2}=&\Big(\int_{\bbr^{2n}}|m_j(\xi,\eta)|^2d\xi d\eta\Big)^{1/2}\\
\le& C\big(\int_{1-2^{-j}}^{1-2^{-j-2}}(1-r^2)^{2\la}r^{2n-1}dr\big)^{1/2}\\
\le& C(2^{-2j\la}2^{-j})^{1/2}\\
=&C2^{-j(\la+1/2)}
\end{align*}

To estimate the $\al$-th derivatives, we use the Leibniz's rule to write
$$
\p^{\al}m_j(\xi,\eta)=\sum_{\al_1+\al_2=\al} C_{\al_1}\p^{\al_1}m(\xi,\eta)\p^{\al_2}\psi_j(|(\xi,\eta)|) .
$$
Noticing that 
 $|\p^{\al_1}m(\xi,\eta)|\le C 2^{-j(\la-|\al_1|)}$  and
$\p^{\al_2}\psi_j(|(\xi,\eta)|)\le C2^{j|\al_2|}$,   we  derive the bound
$\p^{\al}m_j(\xi,\eta)$ by $C2^{-j(\la-|\al|)}$.
\end{proof}

The multiplier $m_j$ is not supported in the annulus of radius $2^j$ and one can verify that its Sobolev norm is not as good as would wish.
Actually the norm increases as the number of derivatives is large. So 
a dilation is necessary to apply  Theorem~\ref{04141}.

Let us define $M_j(\xi,\eta)=m_j(2^{-j}\xi,2^{-j}\eta)$, which
is supported in the annulus $\{(\xi,\eta)\in\mathbb R^{2n}: {2^{j}-1} \le |(\xi,\eta)| \le   {2^{j}-1/4} \}$, whose width is $3/4$. Based on Lemma \ref{om},
we have the following corollary.
\begin{cor}\label{mm}%Modified Multiplier
The multipliers $M_j(\xi,\eta)=m_j(2^{-j}\xi,2^{-j}\eta)$ satisfy
$$
\|\p^{\al}M_j\|_{L^{\nf}}\le C2^{-j\la} \qq\text{for all multiindex }\al, 
$$
$$\nabla m_j(\xi,\eta)=2^j(\nabla M_j)(2^j\xi,2^j\eta),$$
and
$$
\|M_j\|_{L^r_s}\le C2^{-j\la}2^{j(2n-1)/r},
$$
where $\|M_j\|_{L^r_s}= \|(I-\Delta)^{s/2} M_j\|_{L^r}$ is the Sobolev norm of $M_j$.

\end{cor}

\begin{proof}
\begin{comment}
A simple change of variables implies that
\begin{align*}
\|M_j\|_{L^2}=&(\int_{\bbr^{2n}}|m_j(2^{-j}\xi,2^{-j}\eta)|^2d\xi d\eta)^{1/2}\\
=& 2^{jn}\|m_j\|_{L^2}\\
\le &C 2^{jn}2^{-j(\la+1/2)}.
\end{align*}
\end{comment}

We have 
$$
| \p^{\al}M_j |\le 2^{-j|\al|}| (\p^{\al}m_j)(2^{-j}\xi,2^{-j}\eta)|\le C
2^{-j|\al|-j(\la-|\al|)}= C2^{-j\la} , 
$$
using \eqref{ofkkri}. 
The verification of the last identity is straightforward
once we notice that $M_j$ is supported in the annulus
$$
\{(\xi,\eta): 2^j-4\le |(\xi,\eta)|\le 2^j-1\}
$$
whose volume is about $2^{j(2n-1)}$.
\end{proof}

 %Let us prove Theorem \ref{Main} using Proposition \ref{De}.
 \begin{proof}[Proof of Theorem \ref{Main}]
 
 It is easy to verify that
 $T_j$ in \eqref{e02261} stays the same if we replace $m_j$
 by $M_j$.
 We apply Theorem~\ref{04141} to $M_j$ with $r=4$,
then
%$$\|T_j\|_{L^2\times L^2\to L^1}\le C2^{-j(\la-\tf{2n+3}4-\ep)}$$
%for any $\ep>0$.
it follows from this and Proposition~\ref{Uni} that $T_*$
is bounded from $L^2\times L^2$ to $L^1$ when $\la>\tf{2n+3}4$.
 \end{proof}
 
Using complex interpolation between Theorem~\ref{Main} and 
Proposition~\ref{e02261} we can obtain a larger range of boundedness,
which we will not pursue here.

 As a  corollary of Theorem \ref{Main}, we obtain the pointwise
 convergence, as $t\to0$, of the operator $A^\la_{t}(f,g)(x)$, which we denote
 by $A_{t}(f,g)(x)$ as well.
 %=A^\la_{t}(f,g)(x)$ defined as\begin{equation}\label{av}%Average\int_{\mathbb R^n}\int_{\mathbb R^n}\wh f(\xi)\wh g(\eta)\big(1-(|t\xi |^2+|t \eta |^2)\big)_+^{\la}e^{2\pi i x\cdot (\xi+\eta)}d\xi d\eta.\end{equation}
 
% {\color{red} Modify the statement of the following corollary}
\begin{prop}
Suppose $\la>\tf{2n+3}4$, %\min\{(4n+3)/5,\ n-1/2\}$, 
then
for $f\in L^{2}$ and $g\in L^{ 2}$  we have
%$$\lim_{t\to\nf}\|T_t(f,g)-fg\|_{L^p},$$and
\begin{equation}\label{pc}%Pointwise convergence
\lim_{t\to0}A_t(f,g)(x)\to f(x)g(x) \qq\text{ a.e..}
\end{equation}
\end{prop}
The proof of this proposition is similar to the linear case, but we sketch it here for completeness.
\begin{proof}
It is easy to establish \eqref{pc} when both $f$ and $g$ are Schwartz functions.
To prove \eqref{pc} for $f\in L^{2}$ and $g\in L^{2}$ it suffices to 
show that for any given $\de>0$ the set
$E_{f,g}(\de)=\{y\in\mathbb R^n:O_{f,g}(y)>\de\}$ has measure $0$,
where
$$
O_{f,g}(y)=\limsup_{\tht\to0}\limsup_{\ep\to0}\big|A_{\tht}(f,g)(y)-A_\ep(f,g)(y)\big|.
$$
For any positive number $\eta$ smaller than $\|f\|_{L^{2}},\, \|g\|_{L^{2}}$, there exist  Schwartz functions $f_1=f-a$ and $g_1=g-b$
such that both $\|a\|_{L^{2}}$, and $\|b\|_{L^{2}}$ are bounded by $\eta$.
We observe that
$$|E_{f,g}(\de)|\le |E_{f_1,g_1}(\de/4)|+|E_{a,g_1}(\de/4)|+|E_{f_1,b}(\de/4)|+|E_{a,b}(\de/4)|.$$
Notice that
$|E_{f_1,g_1}(\de/4)|=0$ since \eqref{pc} is valid for $f_1,g_1$. 
To control the remaining three terms, we observe that, for instance, 
\begin{align*}
|E_{a,g_1}(\de/4)|\le\,\, &\,\,  |\{y:\ 2T_*(a,g_1)(y)>\de/4\}|\\
\le\,\, &\,\, C\f{\|a\|_{L^{2}}\|g_1\|_{L^{2}}}\de \\
\le \,\, &\,\, C  \f{\eta\|g\|_{L^{2}}}\de ,
\end{align*}
where the last term goes to $0$ as $\eta\to0$ since $g$ and $\de$ are fixed.
\end{proof}

\section{Appendix: Proof of Proposition~\ref{BL}}

The proof of this proposition is essentially contained in \cite[Lemma 6]{Grafakos2016a}, 
but for the sake of completeness we include it, ignoring some routine calculations that can be found in \cite{Grafakos2016a}. 
%We indicate here mainly the tiny difference which 
%can be tracked easily.

%So it is easy to obtain Proposition \ref{BL} by examining the proof
%in \cite{Grafakos2015} carefully. We sketch the proof here for the sake of completeness.
\begin{proof}[Proof of Proposition~\ref{BL}]
Notice that in the support of $\nabla M_{j,\ga}^1(\xi,\eta)$, we have $\xi
\in E$ and $\eta\in E$, hence we may alway assume that $\wh f=\wh f\chi E$ and
$\wh g=\wh g\chi_E$. In other words, it suffices to establish 
\eqref{e06061} without $\chi_ E$.

%We choose $r>\max\{4,\tf{2n}\ep\}$.
It suffices to consider, for example, the typical term $\xi_1\p_{\xi_1}M_{j,\ga}^1(\xi,\eta)$, which is
$\sum_k\sum_{l}a_{k,l} 
\p_{\xi_1}\om_{1,k}(\xi)\xi_1\om_{2,l}(\eta)$ for allowed $k,\ l$ in $M^1_{j,\ga}$.
% by Lemma \ref{gra},. 
We rewrite
this as 
\begin{equation}\label{e08311}
2^{j}2^\ga\sum_k\sum_{l}b_{k,l} 
\tilde\om_{1,k}(\xi)\tilde\om_{2,l}(\eta),
\end{equation} 
where 
$\tilde\om_{1,k}(\xi)=2^{-j}\p_{\xi_1}\om_{1,k}(\xi)\xi_1
/\|\p_{\xi_1}\om_{1,k}(\xi)\|_{L^r}$,
$\tilde\om_{2,l}=\om_{2,l}/\|\om_{2,l}\|_{L^r}$,
and $b_{k,l}=2^{-\ga}a_{k,l}\|\p_{\xi_1}\om_{1,k}(\xi)\|_{L^r}\|\om_{2,l}\|_{L^r}$.
%$v_k(\xi)=\p_{\xi_1}\om_{1,k}(\xi)\xi_1$,
%which has the same support as that of $\om_{1,k}$. 

We need some estimates of $\tilde\om_{1,k}$ which will be useful later. The function 
$\p_{\xi_1}\om_{1,k}(\xi)$ is of the form
$2^\ga 2^{\ga n/2}\vp(2^\ga\xi)$ for a compactly supported smooth function $\vp$,
hence $\|\p_{\xi_1}\om_{1,k}(\xi)\|_{L^r}\approx 2^{\ga(1+\tf n2-\tf nr)}$.
This implies that $\|\tilde\om_{1,k}\|_{L^\nf}\le C2^{\ga n/r}$ since $|\xi_1|
\le C2^j$.

We have 
$$
\bigg\|(\sum_{k,l}2^{\ga s}|2^{\ga n}a_{k,l} 
\chi_{Q_{\ga, k,l}}|^2)^{1/2}\bigg\|_{L^r} \leq C\|M_j\|_{L^r_s},
$$
by Lemma~\ref{08311}, %\cite[Theorem 1.64]{Triebel2006},
where $Q_{\ga, k,l}$ is the cube  centered at $2^{-\ga}(k,l)$ with length
$2^{1-\ga}$. 
%Notice that %$|\xi_1|\le 2^j$, and $\p_{\xi_1}\om_{1,k}(\xi)$ is of the form
%$2^\ga 2^{\ga n/2}\vp(2^\ga\xi)$ for a compactly supported smooth function $\vp$,hence
%$2^{-j}2^{-\ga}\p_{\xi_1}\om_{1,k}(\xi)\xi_1\om_{2,l}(\eta)$
%is controlled by 
%$$\sum_{|(k,l)-(k',l')|\le B}2^{\ga n}\chi_{Q_{\ga, k',l'}}$$
%with some absolute constant $B$, so from the previous inequality
%we obtain 
This leads to
$$
\bigg\|(\sum_{k,l}2^{\ga s}|2^{-j}2^{-\ga}a_{k,l}\p_{\xi_1}\om_{1,k}(\xi)\xi_1\om_{2,l}(\eta)|^2)^{1/2}\bigg\|_{L^r} \leq C\|M_j\|_{L^r_s}.
$$
Recall that $\|M_j\|_{L^r_s}\le C2^{-j\la}%2^{j(2n-1)/r}\le C2^{-j(\la-\ep)}
$. Then using the disjointness of supports of $\om_{k,l}$ we obtain further that
\begin{align*}
B=(\sum |b_{k,l}|^r)^{1/r} 
%=& \bigg(\sum_{k,l}\int \Big(|2^{-j}2^{-\ga}a_{k,l}\p_{\xi_1}\om_{1,k}(\xi)\xi_1\om_{2,l}(\eta)|^2\Big)^{r/2}d\xi d\eta\bigg)^{1/r} \\
%& \le \Big\|\Big(\sum_{k,l} |2^{-j}2^{-\ga}a_{k,l}\p_{\xi_1}\om_{1,k}(\xi)\xi_1\om_{2,l}(\eta)|^2\Big)^{1/2}\Big\|_{L^r}\\
%&  \leq C \|M_j\|_{L^r_s} 2^{-s\ga}\\
\le C2^{-j\la}2^{-s\ga}.
\end{align*}
%This estimate is the same as $(\sum |b_{k,l}|^r)^{1/r}$ in the proof of \cite[Lemma 6]{Grafakos2016a}, so the same argument runs exactly the same way.

%without changing anything.

Each  $\omega$ in level $\ga$ is of the form $\omega=\omega_k\omega_l$ 
with $\vu=(k,l)$, where $k$ and $l$ both range over  index sets %$U_1$ and $U_2$ 
of cardinality at most $C2^{jn}2^{\ga n}.$ Moreover we denote by $b_{kl}$ the coefficient $b_\om$, and
we define a bilinear multiplier
$$ \varsigma_\ga=\sum_{k\in U_1}\tilde \omega_k\sum_{l\in U_2} b_{k l} \tilde \omega_l.$$

Let $A$ be a number between $ \|b\|_\nf$ and $B=\|b\|_r$.
Related to $\tau\ge 0$ we define $U_\tau=\{(k,l): 2^{-\tau-1}A\le |b_{k,l}|\le 2^{-\tau}A\}$.
Denote by $col_k=\{(k,l)\in U_\tau:\, k\text{ fixed}\}$.
Define
$$
U_\tau^1=\{(k,l)\in U_\tau:\, \#col_k\ge N_1\},
$$
where $N_1$ is a to be determined number.
So $U_\tau^1$ is a union of long columns.
We denote by $P_1U_\tau^1=\{k: \exists\ l\ s.t.\ (k,l)\in U_\tau^1 \}$, the projection of $U_\tau^1$
onto the $k$-axis. Then the number of columns is $\#P_1U_\tau^1\le B^r(2^{-\tau}A)^{-r}N_1^{-1}:=N_2$.

Let $U_\tau^2$ be the complement of $U_\tau^1$ in $U_\tau$. Associated to $U_\tau^i$ we can define a
 bilinear multiplier $\varsigma_\tau^i=2^{j}2^\ga\sum_{(k,l)\in U_\tau^i}b_{k,l}\tilde \om_{k,l}$, and
a bilinear operator $T_{\varsigma_\tau^i}$.
A well-known argument (see, for instance, \cite{Grafakos2015} or \cite{Grafakos2016a}) shows that
$$
\|T_{\varsigma_\tau^1}(f,g)\|_{L^1}\le C 2^{j}2^\ga N_2^{1/2} 2^{2\ga n/r}2^{-\tau}A \|f\|_{L^2}\|g\|_{L^2}
$$
and
$$
\|T_{\varsigma_\tau^2}(f,g)\|_{L^1}\le C 2^{j}2^\ga N_1^{1/2} 2^{2\ga n/r}2^{-\tau}A \|f\|_{L^2}\|g\|_{L^2}.
$$
Identifying $N_1$ and $N_2$, and taking $A=B$ in our situation, 
we obtain that $N_1=N_2=C2^{\tau r/2}$, which implies that 
the 
$\|T_{\varsigma_\tau^i}\|_{L^2\times L^2\to L^1}$  is bounded by 
%{\color{red}
$C2^{-j(\la-1)}2^{-\ga(s-\tf {2n}r-1)}2^{-\tau(1-\tf r4)}$.
%}

Summing over $\tau$, we obtain the claimed bound for $r<4$.

%\bq $r=4$?

For the case $r=4$, we may assume that  $\tau\le\tau_m=2(j+\ga)n/4$ since $N_2=2^{\tau r/2}\le 2^{(j+\ga)n}$ with $r=4$. Actually we define 
$$U_{\tau_m}=\{(k,l):\, |b_{k,l}|\le 2^{-\tau_m}A\}.$$
Then the previous argument gives the bound 
%{\color{red}
$C (j+\ga)n 2^{-j(\la-1)}2^{-\ga(s-\tf n2-1)}$ when $r=4$.
%}
\end{proof}

A lemma concerning the decay of the coefficients related to the orthonormal basis in 
Lemma \ref{wave} is given below.
\begin{lm}[\cite{Grafakos2015}]\label{smooth}%Smoothness implies decay
 Suppose $\varsigma(\xi,\eta)$ defined on $\mathbb R^{2n}$ satisfies that there exists a constant $C_M$ such that 
%  there exists a constant $C_{\al}>0$ such that  
$\|\p^\al(\varsigma(\xi,\eta))\|_{L^{\nf}}\le C_{M}$ for each multiindex  $|\al|\le M$,
where $M$ is the number of vanishing moments of $\psi_M$. %, where $M$ is a large enough integer.
%using the notations in Lemma \ref{wave}, 
Then for %any $j\in \mathbb Z$ and 
any 
nonnegative
integer $\ga \in \mathbb N_0=\{n\in\mathbb Z: n\ge 0\}$ we have 
\begin{equation}\label{887}
|\langle \Psi^{\ga,G}_{\vec \mu}, \varsigma\rangle| \leq C C_M2^{-(M+n)\ga} \, .
\end{equation}
% and $\de$ is as in \eqref{del9}. 
\end{lm}

%\bq *Does the constant $C$ depend on the support of $\varsigma$?

%{\color{red} \bq Check Appendix B.2 in \cite{Grafakos2014a} where $M=0$ and $L=0$.}

This lemma can be proved by applying Appendix B.2 in \cite{Grafakos2014a}, and we delete the  
details 
which can be found in \cite{Grafakos2015}.

\medskip

%{\color{red}
 
By this lemma
we have a better decay in $j$ for $b_{k,l}$ compared with Corollary~\ref{06061}, 
namely $|b_{k,l}|\le C2^{-ja}2^{-\ga (s+n)}$, using 
$|\p^\be m|\le C|(\xi,\eta)|^{-a}$ in Theorem~\ref{04071} if we assume 
$s$ number of derivatives.
It is natural to conjecture that this better decay in $j$ can  lower the restriction on $a$.
This, unfortunately, is not true.

As we did before,
setting $N_1=N_2$ implies that
$N_1=2^{\tau r/2}$.
An important observation is that $|b_{k,l}|\ll B$. Actually the smallest
$\tau$ such that $2^{-\tau}B\sim  \|b_{k,l}\|_{\ell^\nf}\le C2^{-ja}2^{-\ga (s+n)}$ is $\tau_0=\tf{2nj}r+n\ga$, which means that the summation 
in $\tau$ starts from $\tau_0$ other than $0$.

Another observation is that $N_2$ related to $\tau_0$ is 
$2^{\tau_0r/2}\sim 2^{nj+ {n\ga r}/2}$, which is smaller than 
$2^{\tau_mr/2}\sim 2^{nj+n\ga}$ when $r>2$. So for $r\in(2,4)$, we 
take $N_2=2^{nj+n\ga}.$
And the summation in $\tau$ consists just one term $\tau_0$.

By the calculation in the proof of Proposition~
\ref{BL} the norm of $T_{\varsigma_\ga}\sum_\tau T_{\varsigma_\tau^i}$, which consists of one term with $\tau=\tau_0$,
is bounded by a constant multiple of $2^{-j(a-  n/2-1)}2^{-\ga (s-  n/2-1)}$.
This provides no  new information except for a bound independent of $r$,
which is natural since there is no $r$  in the conditions of Theorem~\ref{04071}.

So we still have the restriction $a>\tf n2+1$.
It is also easy to verify that when $r=4$ the bound for $T_{\varsigma_\tau^i}$ does not change, so in this case we need  $a>\tf n2+1$ as well.

%}

%\end{rmk}

\begin{rmk}
We use mainly the case $r=4$ in applying Proposition~\ref{BL}, while a smaller $A$, which reduces the number of
$\tau$'s involved, does not change the exponential decay in $j$ at all.

\end{rmk}

\begin{rmk}

Lemma~\ref{smooth} implies also a better decay of the off-diagonal
part in $j$, namely $2^{-j(a-1/2)}$, which, however, is useless
for us due to the restriction of the diagonal part. 
\end{rmk}

\begin{comment}
\section{Maybe}

\subsection{Manifolds beyond sphere}

\cite{Greenleaf1981}
\subsection{Bilinear Schr\"odinger maximal }

\cite{Carbery1985}
\subsection{$L^r$ Hormander}

\cite{Kurtz1979}

\end{comment}

\end{document}